\documentclass[12pt]{article}
\usepackage{amsmath}
\usepackage{amssymb}
\usepackage{amsthm}
\usepackage{mathrsfs}
\usepackage{comment}
\usepackage[capitalise]{cleveref}
\crefformat{equation}{(#2#1#3)}
\usepackage{enumerate}
\usepackage{graphicx}
\usepackage[margin=1.5in]{geometry}

\newtheorem{thm}{Theorem}
\numberwithin{thm}{section}
\newtheorem*{thm*}{Theorem}
\newtheorem{qst}[thm]{Question}
\newtheorem*{qst*}{Question}
\newtheorem*{oq*}{Open Question 3.10}
\newtheorem{prp}[thm]{Proposition}
\newtheorem{lma}[thm]{Lemma}
\newtheorem{cor}[thm]{Corollary}

\newtheorem{clm}{Claim}
\newtheorem{sclm}{Subclaim}
\newcommand{\mo}{<_M}

\newcommand{\wo}{\leq_S}

\theoremstyle{definition}
\newtheorem{defn}[thm]{Definition}

\newtheorem{rmk}[thm]{Remark}

\makeatletter\@addtoreset{case}{thm}\makeatother
\makeatletter\@addtoreset{clm}{thm}\makeatother
\makeatletter\@addtoreset{sclm}{thm}\makeatother
\begin{document}
\title{The Linearity of the Mitchell Order}
\author{Gabriel Goldberg}
\maketitle
\section{Introduction}
The subject of this paper is an axiom inspired by the theory of canonical inner models for large cardinal hypotheses and its implications for the structure of the Mitchell order. This axiom, the {\it Ultrapower Axiom}, holds in all known canonical inner models, and despite its simplicity, seems to distill many of the structural features of these models with respect to the class of countably complete ultrafilters, the basic building blocks of modern large cardinals. 

The Ultrapower Axiom follows from a very weak form of the comparison lemma (Woodin's {\it Weak Comparison}), and should itself be viewed as a very weak comparison principle. The argument that these weak comparison principles hold in canonical inner models is so general that they seem likely to hold in any inner model constructed by anything close to the current methodology. The investigation of the Ultrapower Axiom therefore serves both to illuminate the structure of the known inner models and to predict the structure of the inner models yet to be discovered. Moreover, a refutation of the Ultrapower Axiom from any large cardinal hypothesis whatsoever would be a strong anti-inner model theorem.

In this paper, we will focus on the implications of the Ultrapower Axiom for the Mitchell order on normal ultrafilters, supercompactness measures, and huge measures. The Ultrapower Axiom offers a new perspective on the structure of the Mitchell order in canonical inner models, so we begin by describing the original perspective. The Mitchell order was first isolated by Mitchell in the context of the models \(L[\mathcal U]\) constructed from coherent sequences of ultrafilters \(\mathcal U\). In these models, the Mitchell order is manifestly a linear order. We outline Mitchell's proof. One first shows by a comparison argument that every normal ultrafilter \(U\) in \(L[\mathcal U]\) is indexed on the sequence \(\mathcal U\). The fact that \(\mathcal U\) is coherent then implies that if \(U_0\) and \(U_1\) are normal ultrafilters on \(\mathcal U\) and the index of \(U_0\) precedes that of \(U_1\), then \(U_0\) precedes \(U_1\) in the Mitchell order.

The linearity of the Mitchell order on normal measures is the simplest feature of canonical inner models that has not been replicated by forcing. Therefore a key test question for the theory of inner models for large cardinal axioms is whether the linearity of the Mitchell order is compatible with very large cardinals. For example, the following question was raised very early on by Solovay-Reinhardt-Kanamori \cite{Kanamori} in a slightly weaker form:

\begin{qst}\label{oldq}
Assume there is a cardinal \(\kappa\) that is \(2^\kappa\)-supercompact. Can the Mitchell order linearly order the normal ultrafilters on \(\kappa\)?
\end{qst}

Woodin \cite{Woodin} and Neeman-Steel \cite{NeemanSteel} have constructed canonical inner models at the finite levels of supercompactness under iteration hypotheses, and so one would expect to dispense easily with this question. Yet Mitchell's argument cannot be generalized to these models: in order to develop a comparison theory for these models, one must explicitly {\it prevent} certain normal measures from appearing on their extender sequences. It turns out, however, that the existence of a comparison theory, rather than any specific requirements about the form of a coherent extender sequence, ensures the linearity of the Mitchell order by a completely different argument. 

\begin{thm}[Ultrapower Axiom]\label{normal}
The Mitchell order on normal ultrafilters is linear.
\end{thm}

Since the Ultrapower Axiom holds in the Woodin and Neeman-Steel models, this answers \cref{oldq} positively under a very plausible iteration hypothesis. Furthermore, the result extends to all generalized normal ultrafilters; we state some of the generalizations below. 

The paper is organized as follows. In the first section we state the Ultrapower Axiom and quickly to prove \cref{normal}. In the next section, we generalize it to a much wider class of ultrafilters, the Dodd solid ultrafilters. Finally, in the last section we prove a generalization of Solovay's lemma that implies the Dodd solidity of generalized normal ultrafilters under a cardinal arithmetic assumption that is necessary. This pushes the linearity of the Mitchell order into the realm of very large cardinals. In particular, we have the following two theorems, though in fact what we will prove is much stronger:

\begin{thm}[Ultrapower Axiom]
Suppose that \(\lambda\) is a cardinal such that \(\lambda = 2^{<\lambda}\) and \(\kappa \leq \lambda\) is \(\lambda\)-supercompact. Then the Mitchell order wellorders the normal fine \(\kappa\)-complete ultrafilters on \(P_\kappa(\lambda)\).
\end{thm}

In a separate paper we show that if \(\kappa\) is supercompact and the Ultrapower Axiom holds, then for all \(\delta\geq \kappa\), \(2^\delta=\delta^+\). This result and its local refinements to a certain extent justify the cardinal arithmetic assumptions in the theorems of this paper.

\begin{thm}[Ultrapower Axiom]
Suppose that \(\lambda\) is a regular cardinal such that \(\lambda = 2^{<\lambda}\). Then the Mitchell order wellorders the normal fine ultrafilters on \(P_\lambda(\lambda)\).
\end{thm}

The assumption of regularity can essentially be dropped, yielding the linearity of the appropriate variant of the Mitchell order on all normal fine ultrafilters assuming GCH, see \cref{invariant}. The statement of this theorem involves the definition of that variant, see \cref{invariantdef}. We note that its proof requires the generalization of Solovay's lemma to singular cardinals, see \cref{optimal}, and this involves some interesting combinatorics.  

\section{The Ultrapower Axiom}
\begin{defn}
Suppose \(M\) is an inner model. An {\it internal ultrafilter} of \(M\) is a set \(W\in M\) such that in \(M\), \(W\) is a countably complete ultrafilter. 
\end{defn}

\begin{defn}
Suppose \(U_0\) and \(U_1\) are countably complete ultrafilters. A {\it comparison} of \(\langle U_0,U_1\rangle\) by internal ultrafilters is a pair \(\langle W_0,W_1\rangle\) such that the following hold. 
\begin{enumerate}[(1)]
\item \(W_0\) is an internal ultrafilter of \(\text{Ult}(V,U_0)\).
\item \(W_1\) is an internal ultrafilter of \(\text{Ult}(V,U_1)\).
\item \(\text{Ult}(\text{Ult}(V,U_0),W_0) = \text{Ult}(\text{Ult}(V,U_1),W_1)\).
\item \(j_{W_0}\circ j_{U_0} = j_{W_1}\circ j_{U_1}\).
\end{enumerate}
Letting \(N = \text{Ult}(\text{Ult}(V,U_0),W_0)\), we will call \(\langle W_0,W_1\rangle\) a comparison of \(\langle U_0,U_1\rangle\) to a {\it common model} \(N\).
\end{defn}

\begin{defn}[Ultrapower Axiom]
Every pair of countably complete ultrafilters admits a comparison by internal ultrafilters.
\end{defn}

Again, the Ultrapower Axiom holds in all known canonical inner models for large cardinal hypotheses. For example, it holds in the largest canonical inner models that have been unconditionally constructed from a large cardinal hypothesis, in the realm of Woodin limits of Woodin cardinals. The more countably complete ultrafilters there are, the more interesting the Ultrapower Axiom becomes, which explains our focus on the Ultrapower Axiom in the context of supercompact cardinals. 

The constructions of canonical inner models conditioned on iterability hypotheses reach the finite levels of supercompactness. For example, the following result shows that the Ultrapower Axiom is almost certainly compatible with a cardinal \(\kappa\) that is \(2^\kappa\)-supercompact.

\begin{thm}[Woodin]\label{finlev}
Assume the \(\mathcal E\)-Iteration Hypothesis. Suppose \(n\) is an natural number and \(\kappa\) is a cardinal that is \(\beth_n(\kappa)\)-supercompact. Then there is an inner model \(M\) in which the following hold.
\begin{enumerate}[(1)]
\item \textnormal{GCH}.
\item There is a cardinal \(\kappa\) that is \(\kappa^{+n}\)-supercompact.
\item The Ultrapower Axiom.
\end{enumerate}
\end{thm}

The following theorem therefore answers the old question (see \cite{Kanamori}) of whether the Mitchell order can linearly order the normal ultrafilters on \(\kappa\) when \(\kappa\) is \(2^\kappa\)-supercompact.

\begin{thm}[Ultrapower Axiom]\label{normal}
The Mitchell order on normal ultrafilters is linear.
\begin{proof}
Suppose \(U_0\) and \(U_1\) are normal ultrafilters. Let \(M_0 = \text{Ult}(V,U_0)\) and \(M_1 = \text{Ult}(V,U_1)\), and let \(j_0 :V\to M_0\) and \(j_1 :V\to M_1\) be the ultrapower embeddings. Without loss of generality, assume that \(\textsc{crt}(j_0) = \textsc{crt}(j_1)\). Let us call this cardinal \(\kappa\). Let \(\langle W_0,W_1\rangle\) be a comparison of \(\langle U_0,U_1\rangle\) by internal ultrafilters to a common model \(N\). Thus there are internal ultrapower embeddings \(k_0 : M_0 \to N\) and \(k_1 : M_1\to N\) given by \(W_0\) and \(W_1\) such that \(k_0\circ j_{0} = k_1\circ j_{1}\). We may assume by symmetry that \(k_0(\kappa) \leq k_1(\kappa)\). 

Suppose first that \(k_0(\kappa) = k_1(\kappa)\). We claim that \(U_0 = U_1\). This is a consequence of the following calculation:
\begin{align*}X\in U_0 &\iff \kappa \in j_{0}(X)\\
&\iff k_0(\kappa) \in k_0(j_{0}(X))\\
&\iff k_1(\kappa) \in k_1(j_{1}(X))\\
&\iff \kappa \in j_{1}(X)\\
&\iff X \in U_1
\end{align*}
The second equivalence follows from the fact that \(k_0\circ j_{0} = k_1\circ j_{1}\) combined and the fact that \(k_0(\kappa) = k_1(\kappa)\).

Suppose instead that \(k_0(\kappa) < k_1(\kappa)\).  We claim that \(U_0\in \text{Ult}(V,U_1)\).  This is a consequence of the following calculation: for any \(X\subseteq \kappa\),
\begin{align*}X\in U_0 &\iff \kappa \in j_{0}(X)\\
&\iff k_0(\kappa) \in k_0(j_{0}(X))\\
&\iff k_0(\kappa) \in k_1(j_{1}(X))\\
&\iff k_0(\kappa) \in k_1(j_{1}(X))\cap k_1(\kappa)\\
&\iff k_0(\kappa) \in k_1(j_{1}(X)\cap \kappa)\\
&\iff k_0(\kappa) \in k_1(X)
\end{align*}
The fourth equivalence follows from the fact that \(k_0(\kappa) < k_1(\kappa)\). Since \(k_1\) is definable over \(M_1\), this calculation shows how to compute \(U_0\) within \(M_1\).
\end{proof}
\end{thm}

\section{Dodd solid ultrafilters}
\begin{defn}
The canonical wellorder of \([\text{Ord}]^{<\omega}\) is defined by \(a < b\) if \(\max ((b\setminus a)\cup (a\setminus b)) \in b\).
\end{defn}

Perhaps it is simpler to identify finite sets of ordinals \(a\in [\text{Ord}]^{<\omega}\) with descending sequences of ordinals, in which case the canonical wellorder is just the lexicographic order.

\begin{defn}
Suppose \(a\in [\text{Ord}]^{<\omega}\). Then \([a] = \{b \in [\text{Ord}]^{<\omega} : b < a\}\).
\end{defn}

\begin{defn}
Suppose \(a\in [\text{Ord}]^{<\omega}\). A countably complete ultrafilter \(U\) on \([a]\) is called a {\it uniform ultrafilter} if for all \(b < a\), \([b]\notin U\). In this case, \(a\) is called the {\it space} of \(U\) and is denoted by \(\textsc{sp}(U)\).
\end{defn}

We now put down the most important definition in the context of the Ultrapower Axiom and the key theorem regarding it, though we will neither prove nor use this result here. We will use the definition, but only in a superficial way. We note that it is motivated by the attempt to generalize \cref{normal}: recall that the key to the proof was to consider, given a comparison \(\langle W_0,W_1\rangle\) of a pair of normal ultrafilters \(\langle U_0, U_1\rangle\), whether \(j_{W_0}(\kappa) \leq j_{W_1}(\kappa)\).

\begin{defn}
The {\it seed order} is the binary relation \(\wo\) defined on uniform ultrafilters \(U_0\) and \(U_1\) by \(U_0\wo U_1\) if and only if there exists a comparison \(\langle W_0,W_1\rangle\) of \(\langle U_0,U_1\rangle\) by internal ultrafilters such that \(j_{W_0}([\text{id}]_{U_0}) \leq j_{W_1}([\text{id}]_{U_1})\).
\end{defn}

\begin{thm}
The following are equivalent.
\begin{enumerate}[(1)]
\item The Ultrapower Axiom.
\item The seed order wellorders the class of uniform ultrafilters.
\end{enumerate}
\end{thm}

The fact that one can {\it define} a wellorder of all uniform ultrafilters assuming a principle as general as the Ultrapower Axiom is quite surprising. For example, it has the following immediate consequence, which is not obvious from the statement of the Ultrapower Axiom.

\begin{cor}[Ultrapower Axiom]
Every uniform ultrafilter is ordinal definable.
\end{cor}

One can just as easily define the seed order for uniform ultrafilters on ordinals. The entire theory lifts through the canonical isomorphism between \(\text{Ord}\) and \([\text{Ord}]^{<\omega}\). The utility of working with sequences of ordinals should be clear from the following definitions. 

\begin{defn}
Suppose \(j:V\to M\) is an elementary embedding and \(c\) is a finite set of ordinals. Let \(a\in [\text{Ord}]^{<\omega}\) be least such that \(c\leq j(a)\). The {\it extender of \(j\) below \(c\)} is the function \(E : P([a]) \to V\) defined by \[E(X) = j(X) \cap [c]\]
If \(U\) is an ultrafilter, then \(U|c\) denotes the extender of \(j_U: V\to \text{Ult}(V,U)\) below \(c\).
\end{defn}

We remark that the extender of \(j\) below \(c\) is not in general an extender in the standard sense, but rather a finite collection of extenders. (The issue is that \([a]\) may not be closed under finite unions.) 

\begin{defn}
Suppose \(U\) is a nonprincipal uniform ultrafilter on \([a]\). Then \(U\) is {\it Dodd solid} if letting \(c = [\text{id}]_U\), \(U | c\in \text{Ult}(V,U)\).
\end{defn}

Thus an ultrafilter is Dodd solid if its ultrapower contains the longest possible initial segment of its extender. For example, every normal ultrafilter is trivially Dodd solid. More interestingly, if \(U\) is indexed on the extender sequence of an iterable Mitchell-Steel model satisfying ZFC, then \(U\) is Dodd solid (by a theorem of Steel). 

\begin{thm}[Ultrapower Axiom]\label{dodd}
The Mitchell order on Dodd solid ultrafilters is linear.
\end{thm}

We will use in the proof the following basic fact about the seed order. We remark that this property is not shared by the Mitchell order. 

\begin{lma}\label{spaces}
Suppose \(U_0\) and \(U_1\) are uniform ultrafilters such that \(U_0\wo U_1\). Then \(\textsc{sp}(U_0) \leq \textsc{sp}(U_1)\).
\begin{proof}
Let \(a_0= \textsc{sp}(U_0)\) and \(a_1 = \textsc{sp}(U_1)\). Since \(U_0\) is uniform, \(a_0\) is the least \(a\in [\text{Ord}]^{<\omega}\) with the property that \([\text{id}]_{U_0}\in j_{U_0}([a])\). Thus to prove the lemma, it suffices to show that \([\text{id}]_{U_0}\in j_{U_0}([a_1])\). 

Of course, \([\text{id}]_{U_1}\in j_{U_1}([a_1])\). Let \(\langle W_0,W_1\rangle\) be a comparison of \(\langle U_0,U_1\rangle\) witnessing \(U_0\wo U_1\). Then \(j_{W_0}([\text{id}]_{U_0}) \leq j_{W_1}([\text{id}]_{U_1})\), but also \[j_{W_1}([\text{id}]_{U_1})\in j_{W_1}(j_{U_1}([a_1])) = j_{W_0}(j_{U_0}([a_1]))\] Hence \(j_{W_0}([\text{id}]_{U_0})\in j_{W_0}(j_{U_0}([a_1]))\) since \(j_{W_0}(j_{U_0}([a_1]))\) is downwards closed in the canonical wellorder on \([\text{Ord}]^{<\omega}\). But then by the elementarity of \(j_{W_0}\), \([\text{id}]_{U_0}\in j_{U_0}([a_1])\) as desired.
\end{proof}
\end{lma}

\begin{proof}[Proof of \cref{dodd}]
Suppose \(U_0\) and \(U_1\) are Dodd solid ultrafilters. Let \(M_0 = \text{Ult}(V,U_0)\) and \(M_1 = \text{Ult}(V,U_1)\), and let \(j_0 :V\to M_0\) and \(j_1 :V\to M_1\) be the ultrapower embeddings. Let \(c_0 = [\text{id}]_{U_0}\) and \(c_1 = [\text{id}]_{U_1}\). Assume without loss of generality that \(U_0 \wo U_1\), and let \(\langle W_0,W_1\rangle\) be a comparison of \(\langle U_0,U_1\rangle\) by internal ultrafilters to a common model \(N\) witnessing this. Thus there are internal ultrapower embeddings \(k_0 : M_0 \to N\) and \(k_1 : M_1\to N\) given by \(W_0\) and \(W_1\) such that \(k_0\circ j_{0} = k_1\circ j_{1}\) and \(k_0(c_0) \leq k_1(c_1)\). 

If \(k_0(c_0) = k_1(c_1)\), then just as in \cref{normal}, one shows that \(U_0 = U_1\). (This is essentially the proof of the antisymmetry of the seed order.) If on the other hand \(k_0(c_0) < k_1(c_1)\), we claim that \(U_0\in \text{Ult}(V,U_1)\). Let \(E = U_1|c_1\) denote the extender of \(U_1\) below \(c_1\). Since \(U_1\) is Dodd solid, \(E\in \text{Ult}(V,U_1)\). 
 If \(X\subseteq [\textsc{sp}(U_0)]\),
\begin{align*}X\in U_0 &\iff c_0 \in j_{0}(X)\\
&\iff k_0(c_0) \in k_0(j_{0}(X))\\
&\iff k_0(c_0) \in k_1(j_{1}(X))\\
&\iff k_0(c_0) \in k_1(j_{1}(X))\cap [k_1(c_1)]\\
&\iff k_0(c_0) \in k_1(j_{1}(X)\cap [c_1])\\
&\iff k_0(c_0) \in k_1(E(X))
\end{align*}
The fourth equivalence follows from the fact that \(k_0(c_0) < k_1(c_1)\). The last equivalence requires \(P(\textsc{sp}(U_0))\subseteq \text{dom}(E)\), which we verify as follows: by \cref{spaces}, \(\textsc{sp}(U_0)\leq \textsc{sp}(U_1)\), and since \(P(\textsc{sp}(U_1)) = \text{dom}(E)\), it follows that \(P(\textsc{sp}(U_0))\subseteq \text{dom}(E)\). Since \(k_1\) is definable over \(M_1\) and \(E\in M_1\), this calculation shows how to compute \(U_0\) within \(M_1\).
\end{proof}

\section{Solovay's Lemma and Singular Cardinals}
The following remarkable theorem, due to Solovay, implies that if \(\lambda\) is a regular cardinal then any normal fine ultrafilter on \(P(\lambda)\) is Rudin-Keisler equivalent to a canonical ultrafilter on \(\lambda\) via the \(\sup\) function. 
\begin{thm}[Solovay's Lemma]\label{sollemma}
Suppose \(\lambda\) is a regular uncountable cardinal and \(\langle S_\alpha :\alpha <\lambda\rangle\) is a partition of \(\textnormal{cof}(\omega)\cap \lambda\) into stationary sets. If \(j:V\to M\) is an elementary embedding of \(V\) into an inner model \(M\) with \(j[\lambda]\in M\), then \(j[\lambda]\) is definable in \(M\) from the parameters \(j(\langle S_\alpha :\alpha <\lambda\rangle)\) and \(\sup j[\lambda]\).
\end{thm}

The key corollary of \cref{sollemma} makes no mention of the arbitrary stationary partition \(\langle S_\alpha :\alpha <\lambda\rangle\).

\begin{cor}\label{solcor}
Suppose \(\lambda\) is a regular uncountable cardinal and \(\mathcal U\) is a normal fine ultrafilter on \(P(\lambda)\). Then \(\mathcal U\) is Rudin-Keisler equivalent to the ultrafilter \[U = \{X \subseteq \lambda : {\sup}^{-1}[X] \in \mathcal U\}\]
\end{cor}

An easy corollary is the following:

\begin{cor}
Suppose \(\lambda\) is a regular cardinal such that \(2^{<\lambda} = \lambda\). Suppose \(\mathcal U\) is a normal fine ultrafilter on \(P(\lambda)\). Then \(\mathcal U\) is Rudin-Keisler equivalent to a Dodd solid ultrafilter on \(\lambda\).
\end{cor}

We omit the proof here, and instead prove a generalization in \cref{normalsolid}. When it exists, we denote the (unique) Dodd solid ultrafilter associated to a normal fine ultrafilter \(\mathcal U\) by \(U_\mathcal U\).  We note that we already have the following consequence of the Ultrapower Axiom and Solovay's lemma. (The restriction to normal fine ultrafilters on \(P_\lambda(\lambda)\) entails no loss of generality by Kunen's inconsistency theorem.)

\begin{cor}[Ultrapower Axiom]
Suppose \(\lambda\) is a regular cardinal such that \(2^{<\lambda} = \lambda\). Then the Mitchell order wellorders the normal fine ultrafilters on \(P_\lambda(\lambda)\).
\begin{proof}
Suppose \(\mathcal U_0\) and \(\mathcal U_1\) are normal fine ultrafilters on \(P_\lambda(\lambda)\). Let \(U_0 = U_{\mathcal U_0}\) and \(U_1 = U_{\mathcal U_1}\). By \cref{dodd}, either \(U_0 \mo U_1\), \(U_1\mo U_0\), or \(U_0 = U_1\). In the latter case, it is easy to see that \(\mathcal U_0 = \mathcal U_1\). Thus assume without loss of generality that \(U_0 \mo U_1\), or equivalently that \(U_0 \in \text{Ult}(V,\mathcal U_1)\). Since \(\lambda\) is a regular cardinal and \(2^{<\lambda} = \lambda\), \(\lambda^{<\lambda} = \lambda\). Therefore since \(\text{Ult}(V,\mathcal U_1)\) is closed under \(\lambda\)-sequences, \(P(P_\lambda(\lambda))\subseteq \text{Ult}(V,\mathcal U_1)\) and the Rudin-Keisler reduction \(f:\lambda \to P_\lambda(\lambda)\) reducing \(\mathcal U_0\) to \(U_0\) is in \(\text{Ult}(V,\mathcal U_1)\). Since \[\mathcal U_0 = \{X\in P(P_\lambda(\lambda)) : f^{-1}(X)\in U_0\}\] and \(P(P_\lambda(\lambda))\), \(f\), and \(U_0\) are in \(\text{Ult}(V,\mathcal U_1)\), \(\mathcal U_0\) is in \(\text{Ult}(V,\mathcal U_1)\).
\end{proof}
\end{cor}

In the case that \(\lambda\) is singular, the seemingly trivial issue of whether the powerset of the space of \(\mathcal U_0\) lies in \(\text{Ult}(V,\mathcal U_1)\) will actually block the attempt to easily state some of our theorems about normal measures on \(P(\lambda)\) in terms of the Mitchell order. There is a slightly deeper issue here, which is that in general when \(\lambda\) is a singular cardinal, given a normal fine ultrafilter \(\mathcal U\) on \(P(\lambda)\), there seems to be no canonical choice of a subset \(X\) of \(P(\lambda)\) to which one can restrict \(\mathcal U\) in order to ensure that \(\mathcal U\restriction X\) is a uniform ultrafilter (in the standard sense of the word uniform); see \cref{counterintuitive} and the comments following it. In the regular case, \(P_\lambda(\lambda)\) works, but for singular cardinals \(P_\lambda(\lambda)\) is usually too large.

\cref{optimal} generalizes these theorems to all cardinals, though the following lemma shows that Solovay's theorem \cref{solcor} does not generalize naively to the singular case.

\begin{lma}\label{iota*lemma}
Suppose \(\lambda\) has cofinality \(\iota\) and \(j:V\to M\) is an elementary embedding. Let \(\iota_* = \sup j[\iota]\) and \(\lambda_* = \sup j[\lambda]\). Let \(g_0: \iota\to \lambda\) be the increasing enumeration of any closed cofinal subset \(\lambda\) of order type \(\iota\).
Then the ordinals \(\iota_*\) and \(\lambda_*\) are interdefinable in \(M\) from the parameter \(j(g_0)\). 
\begin{proof}
Note that \(\lambda_* = j(g_0)(\iota_*)\) since \(j\circ g_0[\iota]\) is cofinal in \(j(g_0)[\iota_*]\) and \(j(g_0)\) is continuous. Clearly this defines \(\lambda_*\) from \(\iota_*\) using \(j(g_0)\), but it also defines \(\iota_*\) from \(\lambda_*\) using \(j(g_0)\): \(\iota_*\) is the unique ordinal \(\alpha\) such that \(j(g_0)(\alpha) = \lambda_*\).
\end{proof}
\end{lma}

Thus if \(\lambda\) is a singular cardinal and \(\mathcal U\) is a normal fine ultrafilter on \(P(\lambda)\), the ultrafilter derived from \(j_\mathcal U\) using \(\lambda_*\) is equivalent to the ultrafilter \(W\) derived from \(j_\mathcal U\) using \(\iota_*\), which is {\it not} Rudin-Keisler equivalent to \(\mathcal U\), since \(j_W\) is continuous at \(\iota^+\) while \(j_\mathcal U\) is not. In fact, \(W\) is Rudin-Keisler equivalent to the projection of \(\mathcal U\) to \(P(\iota)\), again by Solovay's lemma.

We state a lemma that is an immediate consequence of Solovay's lemma, just because we will apply it many times in the proof of \cref{optimal}.
\begin{lma}\label{toomany}
Suppose \(i:V\to N\) is an elementary embedding, \(\iota\) is a regular cardinal, and \(i[\iota]\in N\). Then for any \(f:\iota \to V\), \(i\circ f\) is in \(N\) and is definable in \(N\) from \(\sup i[\iota]\) and a point in the range of \(i\).
\begin{proof}
By Solovay's lemma, \(i[\iota]\) is definable in \(N\) from \(\sup i[\iota]\) and a point in the range of \(i\). But \(i\circ f = i(f)\circ i\restriction \iota\).
\end{proof}
\end{lma}

We now prove the correct generalization of Solovay's lemma.

\begin{thm}\label{optimal}
Suppose \(\lambda\) is an uncountable cardinal and \(j\) is an elementary embedding from \(V\) to an inner model \(M\) such that \(j[\lambda]\in M\). Let \(\theta\) be the least generator of \(j\) greater than or equal to \(\sup j[\lambda]\). Then \(j[\lambda]\) is definable in \(M\) from \(\theta\) and a point in the range of \(j\).
\begin{proof}
By Solovay's lemma, we may assume \(\lambda\) is singular. To avoid trivialities, we also assume \(\textsc{crt}(j) < \lambda\). Let \(\iota\) denote the cofinality of \(\lambda\) and \(\iota_*\) denote \(\sup j[\iota]\). Let \(\lambda_*\) denote \(\sup j[\lambda]\). Let \(E\) be the extender of length \(\lambda_*\) derived from \(j\) and let \(M_E = \text{Ult}(V,E)\). Let \(\langle \gamma_\xi : \xi <\iota\rangle\) enumerate a cofinal set of regular cardinals below \(\lambda\). Let \(e : \iota \to \lambda_*\) be the function \(e(\xi) = \sup j[\gamma_\xi]\). Note that \(e\in M\) since \(j[\lambda]\) is in \(M\) and \(\langle \gamma_\xi : \xi <\iota\rangle\) is in \(M\). Let \(J\) be the ideal of bounded subsets of \(\iota\). We state the key observation: 
\begin{clm}\label{eclm} The equivalence class \([e]_J\) of \(e\) modulo \(J\) is definable in \(M\) from \(\lambda_*^{+M_E}\) and a point in the range of \(j\).\end{clm}
As a consequence of the proof of \cref{eclm} we will show the following:
\begin{clm}\label{lambda*}The embedding \(j\) has a generator above \(\lambda_*\) and its least generator \(\theta\) equals \(\lambda_*^{+M_E}\).\end{clm}
Assuming \cref{eclm} and \cref{lambda*}, the following claim completes the proof.
\begin{clm}\label{easyclm}
\(j[\lambda]\) is definable in \(M\) from \([e]_J\) and a point in the range of \(j\).
\begin{proof}[Proof of \cref{easyclm}]
For \(\xi < \iota\), let \(\mathcal T_\xi\) be a stationary partition of \(\text{cof}(\omega)\cap \gamma_\xi\). For any \(e'\in [e]_J\), for all sufficiently large \(\xi_0 < \iota\), \(j[\lambda]\) is the union over \(\xi \in [\xi_0, \iota)\) of the sets \(X_\xi\) obtained by applying Solovay's lemma to \(j(\mathcal T_\xi)\) and \(e'(\xi)\). In this way, \(j[\lambda]\) is definable from \([e]_J\) and \(\langle j(\mathcal T_\xi) :\xi < \iota\rangle\). Obviously \(\lambda_*\) is definable from \([e]_J\), and hence \(\iota_*\) is definable \(\lambda_*\) by \cref{iota*lemma}. Thus \(\langle j(\mathcal T_\xi) :\xi < \iota\rangle\) is definable from \([e]_J\) and a point in the range of \(j\) by \cref{toomany}. Thus \(j[\lambda]\) is definable in \(M\) from \([e]_J\) and a point in the range of \(j\).
\end{proof}
\end{clm}
We turn to the proof of \cref{eclm}, which constitutes the bulk of \cref{optimal}.
\begin{proof}[Proof of \cref{eclm}]
Denote by \(\mathscr D^M\) the product \(\prod_{\xi< \iota} j(\gamma_\xi)\cap M\). Note that this product is in \(M\) since \(\langle j(\gamma_\xi) : \xi < \iota\rangle\in M\) by \cref{toomany}. In fact, for any extender \(F\) derived from \(j\) with length in \((\sup j[\iota],\lambda_*]\), we denote the ultrapower by \(j_F:V\to M_F\), the factor embedding to \(M_E\) by \(k_{FE}: M_F\to M_E\), the factor embedding to \(M\) by \(k_F:M_F\to M\), and the product \(\prod_{\xi< \iota} j_F(\gamma_\xi)\cap M_F\) by \(\mathscr D^{M_F}\). Then \(\mathscr D^{M_F}\in M_F\) by \cref{toomany}. (Note that \cref{toomany} applies in this situation since \(j[\iota]\) is in the hull that collapses to \(M_F\) by Solovay's lemma, and so \(j_F[\iota]\) is in \(M_F\). In fact, \(j_F[\iota] = j[\iota]\), but this is not really relevant.) In particular \(\mathscr D^{M_E} = \mathscr D^M\cap M_E\) since \(j_E(\gamma_\xi) = j(\gamma_\xi)\) for all \(\xi < \iota\). 

We break the proof into two more claims.
\begin{sclm}\label{scaleclm}
In \(M\), there is a \(\lambda_*^{+M}\)-scale in \(\mathscr D^M/J\) that is definable from \(\iota_*\) and a point in the range of \(j\). Moreover, for any such scale \(\langle f_\xi : \xi < \lambda_*^{+M}\rangle\), in \(M_E\), \(\langle f_\xi : \xi < \lambda_*^{+M_E}\rangle\) is a scale in \(\mathscr D^{M_E}/J\).
\end{sclm}
\begin{proof}[Proof of \cref{scaleclm}]
If in \(M\) there is a scale in  \(\mathscr D^M\) of length \(\lambda_*^{+M}\), then there is one definable from \(\iota_*\) and a point in the range of \(j\): first of all, \(\mathscr D^M\) is definable from \(\iota_*\) and a point in the range of \(j\) by \cref{toomany}; second, \(\lambda_*^{+M}\) is definable in \(M\) from \(\iota_*\) and a point in the range of \(j\) by \cref{iota*lemma}; third, the class of points definable in \(M\) from \(\iota_*\) and a point in the range of \(j\) forms an elementary substructure of \(M\) by Los's theorem.

There is a \(\lambda_*^{+M}\)-scale in \(\mathscr D^M/J\) in \(M\) since \(|\mathscr D^M|^M = \lambda_*^{+M}\), combined with the standard fact that \(\mathscr D^M/J\) is \({\leq}\lambda_*\)-directed. We prove \(|\mathscr D^M|^M = \lambda_*^{+M}\). Indeed, \((\lambda_*^\iota)^M = \lambda_*^{+M}\cdot (2^\iota)^M\) by Solovay's theorem that SCH holds above a supercompact applied in \(M\): in \(M\), \(j(\kappa)\) is supercompact to \(\lambda_*\), since in \(V\), \(\kappa\) is supercompact to \(\lambda\) since we assumed \(\kappa < \lambda\). (Here we use Kunen's observation that the condition \(j(\kappa) > \lambda\) can be omitted in the definition of \(\lambda\)-supercompactness using his inconsistency theorem.) But \((2^\iota)^M < \lambda_*\), since in \(M\) there is a strongly inaccessible cardinal between \(\lambda\) and \(\lambda_*\). Towards this, let \(\langle \kappa_n : n < \omega\rangle\) denote the critical sequence of \(j\), and let \(n<\omega\) be least such that \(\lambda < \kappa_{n+1}\). Then \(\kappa_n < \lambda\) since we assumed \(\kappa_0 < \lambda\). Thus \(\kappa_{n+1} < \lambda_*\). Moreover since \(P(\lambda) \subseteq M\), \(\kappa_n\) is inaccessible, and hence \(\kappa_{n+1}\) is inaccessible in \(M\). (That \(P(\lambda)\subseteq M\) follows from \(j[\lambda]\in M\) since for any \(A\subseteq \lambda\), \(A = \{\alpha < \lambda: j(\alpha)\in j(A)\}\).)
\end{proof}

\begin{sclm}\label{Ecofinal}
The function \(e\) is an exact upper bound of \(\mathscr D^{M_E}/J\).
\end{sclm}
\begin{proof}[Proof of \cref{Ecofinal}]
We show first that \(\mathscr D^{M_E}\) is cofinal in \(e\). This follows from \cref{toomany}: by \cref{toomany}, \(j_E\circ f \in M_E\) for all \(f: \iota \to V\). The collection of all \(j_E\circ f\) for \(f\in\prod_{\xi< \iota}\gamma_\xi\) is clearly cofinal in \(e\), recalling that \(j_E\circ f = j\circ f\) for such \(f\).

Now we show that \(e\) is an upper bound of \(\mathscr D^{M_E}/J\). Suppose \(f \in \mathscr D^{M_E}\). For an extender \(F\) derived from \(j\) with length in \((\sup j[\iota],\lambda_*)\) and some \(\bar f\in M_F\), \(f = k_{FE}(\bar f)\) since \(M_E\) is the direct limit of such \(M_F\). By the elementarity of \(k_{FE}\), \(\bar f\in \mathscr D^{M_F}\). Since the length of \(F\) is strictly below \(\lambda_*\), the space of \(F\) is strictly below \(\lambda\). Thus for some \(\xi_0 < \iota\), \(j_F\) is continuous at all regular cardinals  \(\delta \geq \gamma_{\xi_0}\). Since \(\bar f\in \mathscr D^{M_F}\), \(\bar f(\xi) < j_F(\gamma_\xi)\) for \(\xi < \iota\). For \(\xi \in [\xi_0, \iota)\), we may therefore choose \(\alpha_\xi < \gamma_\xi\) such that \(\bar f(\xi) < j_F(\alpha_\xi)\). For \(\xi < \xi_0\), set \(\alpha_\xi = 0\). Let \(h = \langle \alpha_\xi : \xi < \iota\rangle\). Then \(\bar f <_J j_F(h)\), and so \(f = k_{FE}(\bar f) <_J j_E(h) < e\), as desired.
\end{proof}

Using the two subclaims, we prove \cref{eclm}. Fix by \cref{scaleclm} a scale \(\langle f_\xi : \xi < \lambda_*^+\rangle\) in \(\mathscr D^M\) definable from \(\iota_*\) and a point in the range of \(j\). Then \(\langle f_\xi : \xi < \lambda_*^{+M_E}\rangle\) is definable from \(\lambda_*^{+M_E}\) and a point in the range of \(j\) by \cref{iota*lemma}. By \cref{scaleclm}, \(\langle f_\xi : \xi < \lambda_*^{+M_E}\rangle\) is a scale in \(\mathscr D^{M_E}\). Thus \([e]_J\) is definable in \(M\) from \(\langle f_\xi : \xi < \lambda_*^{+M_E}\rangle\) as the equivalence class of any exact upper bound of \(\langle f_\xi : \xi < \lambda_*^{+M_E}\rangle\). \cref{eclm} follows from this and the definability of \(\langle f_\xi : \xi < \lambda_*^{+M_E}\rangle\).
\end{proof}
\begin{proof}[Proof of \cref{lambda*}]
Note that if \(\theta\) exists at all, \(\theta \geq \lambda_*^{+M_E}\), since the least generator of \(j\) above \(\lambda_*\) is a regular cardinal of \(M_E\) greater than or equal to \(\lambda_*\), and \(\lambda_*\) itself is not regular in \(M_E\) by \cref{toomany} with \(f\) set to \(\langle \gamma_\xi :\xi < \iota\rangle\).

To show that \(\lambda_*^{+M_E} = \theta\), it suffices to show that \(k_E(\lambda_*^{+M_E}) \neq \lambda_*^{+M_E}\), or in other words that \(\lambda_*^{+M} \neq \lambda_*^{+M_E}\). That is, we must show \(\lambda_*^{+M_E}\) is not a cardinal in \(M\). This follows from the fact that \(\{j\circ f : f\in \prod_{\xi < \iota} \gamma_\xi\}\) is in \(M\), has cardinality less than \(\lambda_*\) in \(M\) and, modulo \(J\), is cofinally interleaved with the \(<_J\)-increasing sequence \(\langle f_\xi : \xi < \lambda_*^{+M_E}\rangle\), which lies in \(M\).
\end{proof}
This completes the proof of \cref{optimal}.
\end{proof}
\end{thm}
\begin{rmk}By a theorem of Shelah applied in \(M\), one can choose \(\langle \gamma_\xi:\xi < \iota\rangle\) so that \(\prod_{\xi < \iota} \gamma_\xi/J\) has true cofinality \(\lambda^+\) in \(M\). (Without assuming GCH, one may have \(2^\iota > \lambda\) so the naive proof used in \cref{scaleclm} fails.) Returning to the proof of \cref{lambda*}, one shows that \(\text{cf}^M(\lambda_*^{+M_E}) = \lambda^+\).
\end{rmk}

As a consequence of \cref{optimal}, all normal fine ultrafilters \(\mathcal U\) on \(P(\lambda)\) project canonically to ultrafilters on an ordinal, though now the projection is only defined up to \(\mathcal U\)-equivalence. 

\begin{defn}
Suppose \(\lambda\) is an uncountable cardinal and \(\mathcal U\) is a normal fine ultrafilter on \(P(\lambda)\). Then \(U_\mathcal U\) denotes the uniform ultrafilter derived from \(\mathcal U\) using \(\theta\) where \(\theta\) is the least generator of \(\mathcal U\) greater than or equal to \(\sup j_\mathcal U[\lambda]\).
\end{defn}

The following is immediate from \cref{optimal}.

\begin{thm}
Suppose \(\lambda\) is an uncountable cardinal and \(\mathcal U\) is a normal fine ultrafilter on \(P(\lambda)\). Then \(\mathcal U\equiv_\textnormal{RK}U_\mathcal U\).
\end{thm}

The next proposition follows from the proof of \cref{optimal}

\begin{prp}\label{spacelemma}
Suppose \(\mathcal U\) is a normal fine ultrafilter on \(P(\lambda)\). Then \[\textsc{sp}(U_\mathcal U) = \begin{cases}\lambda&\textnormal{if }\textsc{crt}(\mathcal U) \leq \textnormal{cf}(\lambda) \leq \lambda\\
\lambda^+&\textnormal{if }\textnormal{cf}(\lambda)< \textsc{crt}(\mathcal U)\leq \lambda\end{cases}\]
\begin{proof}
We may assume by \cref{solcor} that \(\lambda\) is singular. By \cref{lambda*}, \(U_\mathcal U\) is derived from \(\mathcal U\) using \(\lambda_*^{+M_E}\) where \(\lambda_* = \sup j_\mathcal U[\lambda]\) and \(E\) is the extender of length \(\lambda_*\) derived from \(\mathcal U\).  

Suppose first that \(\textsc{crt}(\mathcal U) \leq \textnormal{cf}(\lambda)\). Clearly \(\textsc{sp}(U_\mathcal U) \geq \lambda\) since \(\lambda_* < \lambda_*^{+M_E}\). But since  \(\textsc{crt}(\mathcal U) \leq \textnormal{cf}(\lambda)\), \(j_\mathcal U\) is discontinuous at \(\lambda\) and so \(\lambda_* < j(\lambda)\). It follows that \(\lambda_*^{+M_E} < j(\lambda)\), and so \(\textsc{sp}(U_\mathcal U) \leq \lambda\). Thus \(\textsc{sp}(U_\mathcal U) = \lambda\).

Suppose instead that \(\textnormal{cf}(\lambda) < \textsc{crt}(\mathcal U)\). Then \(j(\lambda) = \lambda_*\). Since \(\lambda_*^{+M_E}\) is a generator of \(\mathcal U\), \(\lambda_*^{+M_E} \neq j(\lambda^+)\). Thus \(j(\lambda^+) > \lambda_*^{+M_E}\), so \(\textsc{sp}(U_\mathcal U) \leq \lambda^+\). Moreover, if \(\xi < \lambda^+\), then \(j_E(\xi) < \lambda^{+M_E}\) and hence \(j_E(\xi) = j_\mathcal U(\xi)\), since \(\lambda^{+M_E}\) is the critical point of the factor map from \(M_E\) to \(\text{Ult}(V,\mathcal U)\). Thus \(j_\mathcal U(\xi) < \lambda^{+M_E}\). Since \(\xi < \lambda^+\) was arbitrary, it follows that \(\textsc{sp}(U_\mathcal U) = \lambda^+\).
\end{proof}
\end{prp}

\cref{spacelemma} has a counterintuitive corollary.

\begin{cor}\label{counterintuitive}
Suppose \(\lambda\) is an uncountable cardinal and \(\mathcal U\) is a normal fine ultrafilter on \(P(\lambda)\). Then there is a set \(X\in \mathcal U\) such that \(|X| = \lambda^{{<}\textsc{crt}(\mathcal U)}\).
\end{cor}

One might expect that \(\mathcal U\) restricts to a uniform ultrafilter on \(P_\delta(\lambda)\) where \(\delta\) is the least ordinal such that \(j_\mathcal U(\delta) > \lambda\), but by \cref{counterintuitive}, this fails whenever \(\textsc{crt}(\mathcal U) \leq \text{cf}(\lambda) < \delta \leq \lambda\). (This only occurs past a huge cardinal.)

\begin{lma}\label{spsc}
Suppose \(\lambda\) is an uncountable cardinal and \(\mathcal U\) is a normal fine ultrafilter on \(P(\lambda)\). Then \(\textnormal{Ult}(V,\mathcal U)\) is closed under \(\textsc{sp}(U_\mathcal U)\)-sequences.
\begin{proof}
Note that \(\textnormal{Ult}(V,\mathcal U)\) is closed under \(P_\kappa(\lambda)\)-sequences where \(\kappa = \textsc{crt}(\mathcal U)\), since \(j[\lambda]\in \text{Ult}(V,\mathcal U)\) by normality and \(j[P_\kappa(\lambda)]\) is easily computed from \(j[\lambda]\). But by \cref{spacelemma}, \(\textsc{sp}(U_\mathcal U) = |P_\kappa(\lambda)|\).
\end{proof}
\end{lma}

The next proof is really a trivial corollary of what we have already done except for the small amount of extender theory that is needed, and which we spell out in detail.

\begin{thm}\label{normalsolid}
Suppose \(\lambda = 2^{<\lambda}\) and \(\mathcal U\) is a normal fine ultrafilter on \(P(\lambda)\). Then \(U_\mathcal U\) is Dodd solid.
\begin{proof}
Let \(\theta = [\text{id}]_{U_\mathcal U}\), which by definition is the least generator of \(\mathcal U\) above \(\lambda_* = \sup j_\mathcal U[\lambda]\). Let \(E = U_\mathcal U | \theta\), the extender of \(U_\mathcal U\) below \(\theta\), which is of course the same as the extender of \(\mathcal U\) below \(\theta\). We must show that \(E \in \text{Ult}(V, U_\mathcal U)\), or in other words that \(E \in \text{Ult}(V, \mathcal U)\).

First of all, \(E | \lambda_*\in \text{Ult}(V, \mathcal U)\) since \(E | \lambda_*\) is easily computed from the restriction of \(j_\mathcal U\) to \(\bigcup_{\alpha< \lambda} P(\alpha)\), which is in \(\text{Ult}(V, \mathcal U)\) since \(\text{Ult}(V, \mathcal U)\) is closed under \(\lambda\)-sequences and \(\lambda = 2^{<\lambda}\).

This implies \(E\in \text{Ult}(V, \mathcal U)\) by the following argument. Note that \(E\) is the extender of length \(\theta\) derived from \(j_{E|\lambda_*}\) since \(\theta\) is the least generator of \(\mathcal U\) above \(\lambda_*\). But \(j_{E|\lambda_*}\restriction \text{Ult}(V, \mathcal U)\) can be defined over \(\text{Ult}(V, \mathcal U)\) by taking the ultrapower by \(E | \lambda_*\in \text{Ult}(V, \mathcal U)\) which is correctly computed in \(\text{Ult}(V, \mathcal U)\) by closure under \(\lambda\)-sequences. Now \(E\) is easily computed from \(j_{E|\lambda_*}\restriction P(\textsc{sp}(E))\), and \(P(\textsc{sp}(E)) \subseteq P(\textsc{sp}(U_\mathcal U)) \subseteq \text{Ult}(V, \mathcal U)\) by \cref{spsc}.
\end{proof}
\end{thm}

We extend the Mitchell order slightly in order to state our main theorem more simply.

\begin{defn}\label{invariantdef}
The {\it invariant Mitchell order} is the relation \(<^*_{M}\) defined for countably complete ultrafilters \(U_0\) and \(U_1\) by \(U_0 <^*_{M} U_1\) if and only if for some \(U_0'\equiv_{\text{RK}} U_0\), \(U'_0 \mo U_1\).
\end{defn}

\begin{lma}
The invariant Mitchell order is a strict wellfounded partial order on the class of nonprincipal countably complete ultrafilters.
\end{lma}

\begin{defn}
Suppose \(U\) is an ultrafilter and \(X\in U\). Then \(U\restriction X\) denotes the ultrafilter \(\{Y\subseteq X : Y\in U\}\).
\end{defn}

Restricted to normal ultrafilters on \(P(\lambda)\), the uniform Mitchell order is equivalent to a seemingly stronger relation.

\begin{lma}\label{strnn}
Suppose \(U_0 <_M^* U_1\) and for some \(X\in U_0\), \(X\in \textnormal{Ult}(V,U_1)\). Then for some \(Y\in U_0\), \(U_0 \restriction Y\mo U_1\).
\end{lma}

\begin{cor}
Suppose \(\lambda\) is a cardinal and \(\mathcal U_0\) and \(\mathcal U_1\) are normal fine ultrafilters on \(P(\lambda)\) such that \(\mathcal U_0 \mo^* \mathcal U_1\). Then for some \(X\in \mathcal U_0\), \(\mathcal U_0 \restriction X\mo \mathcal U_1\).
\begin{proof}
Note that \(P(\lambda)\in \text{Ult}(V,\mathcal U_1)\) and apply \cref{strnn}.
\end{proof}
\end{cor}

\begin{thm}[Ultrapower Axiom]\label{invariant}
Suppose \(2^{<\lambda} = \lambda\). Then the invariant Mitchell order wellorders the set of normal fine ultrafilters on \(P(\lambda)\).
\begin{proof}
Suppose \(\mathcal U_0\) and \(\mathcal U_1\) are normal fine ultrafilters on \(P(\lambda)\). Let \(U_0 = U_{\mathcal U_0}\) and \(U_1 = U_{\mathcal U_1}\). Since \(U_0\) and \(U_1\) are Dodd solid, either \(U_0 \mo U_1\), \(U_1\mo U_0\), or \(U_0 = U_1\). If \(U_0 = U_1\), it is easy to see that \(\mathcal U_0\) and \(\mathcal U_1\) are equal: normal fine ultrafilters on \(P(\lambda)\) are canonically determined by their ultrapower embeddings. Otherwise assuming without loss of generality that \(U_0\mo U_1\), it is clear that \(U_0 \mo \mathcal U_1\), and so by definition \(\mathcal U_0\mo^* \mathcal U_1\).
\end{proof}
\end{thm}

\bibliography{seedorder}{}

\begin{thebibliography}{1}

\bibitem{Kanamori}
Robert~M. Solovay, William~N. Reinhardt, and Akihiro Kanamori.
\newblock Strong axioms of infinity and elementary embeddings.
\newblock {\em Ann. Math. Logic}, 13(1):73--116, 1978.

\bibitem{Woodin}
W.~Hugh Woodin.
\newblock In search of ultimate-{$L$} the 19th midrasha mathematicae lectures.
\newblock {\em Bull. Symb. Log.}, 23(1):1--109, 2017.

\bibitem{NeemanSteel}
Itay Neeman and John Steel.
\newblock Equiconsistencies at subcompact cardinals.
\newblock {\em Arch. Math. Logic}, 55(1-2):207--238, 2016.

\end{thebibliography}
\bibliographystyle{unsrt}

\end{document}